\documentclass[11pt, a4paper]{amsart}
\usepackage{amssymb, array, amsmath, amscd, pdfpages, enumerate, amsthm, fixltx2e, setspace, hyperref}

\DeclareMathOperator*{\dist}{dist}
\DeclareMathOperator*{\supp}{supp}
\DeclareMathOperator*{\R}{Re}
\DeclareMathOperator*{\I}{Im}
\DeclareMathOperator*{\sgn}{sgn}

\newcommand{\e}{\mathrm{e}}
\newcommand{\ii}{\mathrm{i}}
\newcommand{\dd}{\mathrm{d}}
\newcommand{\Mlog}{M_\mathrm{log}}
\newcommand{\mlog}{m_\mathrm{log}}

\newcommand{\Lloc}{L^1_\mathrm{loc}}
\newcommand{\Lip}{\mathrm{Lip}}
\newcommand{\BUC}{\mathrm{BUC}}
\newcommand{\Cc}{C_\mathrm{c}}
\newcommand{\Cb}{C_\mathrm{b}}
\newcommand{\Wc}{W_\mathrm{c}^{1,1}}
\newcommand{\B}{\mathcal{B}}

\newcommand{\F}{\mathcal{F}}

\newcommand{\RR}{\mathbb{R}}
\newcommand{\CC}{\mathbb{C}}
\newcommand{\NN}{\mathbb{N}}

\newtheorem{thm}{Theorem}[section]

\newtheorem{lem}[thm]{Lemma}
\newtheorem{cor}[thm]{Corollary}
\theoremstyle{definition}
\newtheorem{rem}[thm]{Remark}
\newtheorem{ex}[thm]{Example}

\numberwithin{equation}{section}  

\begin{document}

\title{Quantified versions of Ingham's theorem}

\author{Ralph Chill}
\address{TU Dresden, Institut f\"ur Analysis, 01062 Dresden, Germany}
\email{ralph.chill@tu-dresden.de}

\author{David Seifert}
\address{St John's College, St Giles, Oxford\;\;OX1 3JP, United Kingdom}
\email{david.seifert@sjc.ox.ac.uk}

\begin{abstract}
We obtain quantified versions of Ingham's classical Tauberian theorem and some of its variants by means of a natural modification of Ingham's own simple proof. As corollaries of the main general results, we obtain quantified decay estimates for $C_0$-semigroups. The results reproduce those known in the literature but are both more general and, in one case, sharper. They also  lead to a better understanding of the previously obscure ``fudge factor''  appearing in proofs based on estimating contour integrals.
 \end{abstract}

\thanks{The second author was supported by the EPSRC grant  EP/J010723/1 held by Professors C.J.K.\ Batty (Oxford) and Y.\ Tomilov (Warsaw). He wishes to thank the Analysis group of the TU Dresden for their kind hospitality during his visit in September 2014, during which much of this work was carried out.}

\maketitle

\section{Introduction}\label{intro}

Ingham's theorem  \cite{Ing33} is a cornerstone of modern Tauberian theory. It leads to an elementary proof of the prime number theorem and also forms the starting point for many important results at the intersection of complex analysis and operator theory, such as the ABLV stability theorem for operator semigroups and results of Katznelson-Tzafriri type, with applications in partial differential equations and dynamical systems; see \cite{CT07} for an overview. The version of Ingham's theorem  most suitable for the purposes of this article is stated as Theorem~\ref{Ingham} below. Here, given a Banach space $X$, the Laplace transform $\smash{\widehat{f}}$ of a bounded measurable function $f:\RR_+\to X$ is defined by 
$$\widehat{f}(\lambda)=\int_0^\infty\e^{-\lambda t}f(t)\,\dd t,\quad \R\lambda>0,$$
and the Fourier transform $\F f$ of a function $f\in L^1(\RR;X)$ is defined by 
$$(\F f)(s)=\int_\RR \e^{-\ii st}f(t)\,\dd t,\quad s\in\RR.$$ 
The notation used for classical function spaces is standard. 

\begin{thm}\label{Ingham}
Let $X$ be a Banach space and let $f\in \BUC(\RR_+;X)$. Suppose there exists a function $F\in \Lloc(\RR;X)$ such that 
\begin{equation}\label{extension}
\lim_{\alpha\to0+}\int_\RR \widehat{f}(\alpha+\ii s)\psi(s)\,\dd s=\int_\RR F(s)\psi(s)\,\dd s
\end{equation}
for all $\psi\in \Cc(\RR)$. Then $f\in C_0(\RR_+;X)$.
\end{thm}

The point of departure for all that follows here is a variant of Ingham's original proof of Theorem~\ref{Ingham}, which is presented below in full.  An alternative argument was given by Korevaar in \cite{Ko82}, and indeed Korevaar's proof is now more well-known than Ingham's due to its use of contour integrals involving an ingenious but somewhat obscure ``fudge factor". Whereas Korevaar's proof is based on properties of the Laplace transform, the calculation in \eqref{Parseval} below shows that the function $F$ can be viewed as the distributional Fourier transform of $f$; see also \cite{Ch98}.

\begin{proof}[Proof of Theorem \ref{Ingham}]
Given a function $\phi\in L^1(\RR)$ such that the Fourier transform $\psi:=\F\phi$ of $\phi$ lies in $C_\mathrm{c}(\RR)$, it follows from the dominated convergence theorem and Parseval's identity that 
\begin{equation}\label{Parseval}\begin{aligned} 
f*\phi(t)&=\int_0^\infty f(s)\phi(t-s)\,\dd s\\&=\lim_{\alpha\to0+}\int_0^\infty \e^{-\alpha s}f(s)\phi(t-s)\,\dd s\\&=\lim_{\alpha\to0+}\frac{1}{2\pi}\int_\RR \e^{\ii st} \widehat{f}(\alpha+\ii s) \psi(s)\,\dd s\\&=\frac{1}{2\pi}\int_\RR \e^{\ii st} F(s) \psi(s)\,\dd s
\end{aligned}\end{equation}
for all $t\ge0$.  Hence $f*\phi(t)\to0$ as $t\to\infty$ by the Riemann-Lebesgue lemma. 

Now let $\phi\in L^1(\RR)$ be as above and such that $\int_\RR\phi(s)\,\dd s=1$. For $R>0$, let $\phi_R(t)=R\phi(Rt)$, $t\in\RR$. By uniform continuity of $f$, given any $\varepsilon>0$, there exists $\delta>0$ such that 
$$\begin{aligned}
\|f(t)-f*\phi_R(t)\|&=\left\|\int_\RR (f(t)-f(t-s))\phi_R(s)\,\dd s\right\|\\&\leq \varepsilon \|\phi\|_1+2\|f\|_\infty\int_{|s|\geq R\delta}|\phi(s)|\,\dd s
\end{aligned}$$
for all $t\geq1$ and $R>0$. Here the function $f$ has been extended by zero to $\RR$. Thus $\|f(t)-f*\phi_R(t)\|\to0$ as $R\to\infty$ uniformly for $t\geq1$, and the result follows.
\end{proof}

 Ingham's theorem can be generalised  in a number of directions, for example by assuming  that  $F\in\Lloc(\RR\backslash\Sigma;X)$ and that \eqref{extension} is  satisfied only for all $\psi\in \Cc (\RR\backslash\Sigma )$, where $\Sigma\subseteq\RR$ is some closed set. In what follows, a function $F\in\Lloc(\RR\backslash\Sigma;X)$ satisfying \eqref{extension}  for all $\psi\in \Cc (\RR\backslash\Sigma )$ will be said to be a \emph{boundary function} of the Laplace transform $\smash{\widehat{f}}$ of $f$. Such generalisations have been applied extensively in the study of the asymptotic behaviour of $C_0$-semigroups and other solution families of linear evolution equations; see \cite[Chapter~4]{ABHN11} for an overview, especially when the singular set $\Sigma$ is countable. 

However, even the simple version in Theorem~\ref{Ingham} with $\Sigma=\emptyset$ has interesting applications to operator semigroups. For instance, if $f(t)=T(t)x$, where $T$ is a bounded $C_0$-semigroup with generator $A$ on a Banach space $X$ and $x\in X$, then Theorem~\ref{Ingham} implies that $\|f(t)\|\to0$ as $t\to\infty$ provided $\sigma(A)\cap\ii\RR=\emptyset$. Moreover, noting that 
\begin{equation*}\label{semigroup}
T(t)A^{-1}x=A^{-1}x+\int_0^t T(s)x\,\dd s,\quad t\geq0, 
\end{equation*}
for all $x\in X$,  it follows from Theorem~\ref{Ingham} that, under the same spectral assumption,  $\|T(t)A^{-1}\|\to0$ as $t\to\infty$. The latter fact is of central importance in the study of energy decay for damped wave equations. Given the recent interest in obtaining rates for this energy decay, it is natural to ask whether there exist quantified versions of Ingham's theorem. In particular, the task becomes to estimate the rate of decay of $f(t)$ as $t\to\infty$ given information about how the boundary function $F$  of $\smash{\widehat{f}}$  behaves near the singular points in $\Sigma$ and near infinity. 
 
The three cases of interest here are, in terminology similar to that of \cite{BCT}, when the boundary function $F$ has a `singularity at infinity', when $F$ has a `singularity at zero' and when $F$ has `singularities at zero and infinity'. In the first case we assume that $\Sigma=\emptyset$ and that the boundary function $F(s)$ has controlled growth as $|s|\to\infty$. This case is treated in Section~\ref{infty}. Some  related results in the case where $\smash{\widehat{f}}$ extends analytically beyond the imaginary axis were obtained in \cite{BD08}  by carefully estimating the contour integrals appearing in Korevaar's proof. In the case of a singularity at zero, we assume that $\Sigma=\{0\}$, that the boundary function $F(s)$ has controlled growth as $|s|\to0$ and, by requiring the function  $f$ to be `asymptotically regular', that the behaviour of $F(s)$ for large values of $|s|$ need not be taken into consideration when estimating integrals of the form appearing on the right-hand side of \eqref{Parseval}. This case is treated in Section~\ref{zero}. Some related results in the context of $C_0$-semigroups have recently been obtained recently in \cite{BCT} by an argument that ultimately goes back to Ingham's proof of Theorem~\ref{Ingham}; see also \cite{Se15}. Finally, in the case of singularities at zero and infinity we again assume that $\Sigma=\{0\}$ but this time without the  assumption of asymptotic regularity. Consequently it will be necessary to assume that the boundary function $F(s)$ has controlled growth both as $|s|\to0$ and as $|s|\to\infty$. This case is treated in Section~\ref{zero_infty}. Some related results in this case were obtained in \cite{Ma11} in the setting where $\smash{\widehat{f}}$ has an analytic extension, again by an argument based on contour integrals.

 The purpose of this article is to address all of the above cases in a unified way and to give quantified versions of Ingham's theorem by extending the technique used in \cite{BCT}. In each case we use the general result to deduce estimates on the rate of decay of $C_0$-semigroups. In particular, this approach recovers the results in \cite{BD08} and \cite{Ma11}, and leads to an improvement of the result in \cite{BCT} when $\Sigma=\{0\}$ and $f$ is asymptotically regular. All of the main results are formulated in the more general setting in which the Laplace transform $\smash{\widehat{f}}$ of $f$ is not assumed to possess a holomorphic extension but merely a boundary function. The new method also leads to a better understanding of the previously obscure ``fudge factor''.  
 
Throughout, all Banach spaces are assumed to be complex and the notation will be standard. In particular, $\RR_+=[0,\infty)$, $\RR_-=(-\infty,0]$ and, given a Banach space $X$, $\Cb(\RR_+;X):=C(\RR_+;X)\cap L^\infty(\RR_+;X)$, $\BUC(\RR_+;X)$ denotes the space of bounded and uniformly continuous functions $f:\RR_+\to X$, both endowed with the supremum norm,  and $\Lip(\RR_+;X)$  the space of Lipschitz continuous functions $f:\RR_+\to X$. For $f\in\Lip(\RR_+;X)$, let $$\|f\|_{\Lip}=\sup\left\{\frac{\|f(s)-f(t)\|}{s-t}:s>t\geq0\right\}.$$ Given a closed operator $A$ on $X$, the spectrum of $A$ is denoted by $\sigma(A)$, its resolvent set by $\rho(A)$ and, for $\lambda\in\rho(A)$, we write $R(\lambda,A)$ for the resolvent operator $(\lambda-A)^{-1}$.

\section{Main results}

As in Section~\ref{intro}, let $X$ be a  Banach space, let $f\in L^\infty(\RR_+;X)$ and suppose that the Laplace transform $\smash{\widehat{f}}$ of $f$ admits a boundary function $F\in\smash{\Lloc}(\RR\backslash\Sigma;X)$, where $\Sigma$ is a closed subset of $\RR$. The following sections treat in turn the cases $\Sigma=\emptyset$ with $F$ having a singularity at infinity, $\Sigma=\{0\}$ with $F$ having a singularity at zero, and $\Sigma=\{0\}$ with $F$ having singularities at zero and infinity. 

\subsection{Singularity at infinity}\label{infty}

The main result in this section is Theorem~\ref{Ingham_infty} below, which is a quantified version of Theorem~\ref{Ingham}. The result also extends \cite{BD08} by giving an estimate on the rate of decay even if the boundary function $F$ has only finitely many derivatives. Given a function $M:\RR_+\to[1,\infty)$ and $k\ge1$, define the function $M_k:\RR_+\to(0,\infty)$ by 
\begin{equation}\label{M_k}
M_k(R)=M(R)\big((1+R)^2M(R)\big)^{1/k},\quad R\ge0,
\end{equation}
and let $\Mlog:\RR_+\to(0,\infty)$ be given by 
\begin{equation}\label{M_log}
\Mlog(R)=M(R)\big(\log(1+R)+\log M(R)\big),\quad R\geq0.
\end{equation}
Note that $M_k$, for each $k\ge1$, and $\Mlog$ are strictly increasing, and hence both possess  inverse functions $M_k^{-1}$ and $\Mlog^{-1}$ defined on the ranges of $M_k$ and $\Mlog$, respectively.

\begin{thm}\label{Ingham_infty}
Let $X$ be a Banach space and let $f\in \Cb(\RR_+;X)\cap \Lip(\RR_+;X)$. Suppose that $\smash{\widehat{f}}$ admits a boundary function $F\in\smash{\Lloc}(\RR;X)$. Then $\smash{f}\in C_0(\RR_+;X)$. Moreover, given a continuous non-decreasing function $M:\RR_+\to[1,\infty)$ , the following hold.
\begin{enumerate}[(a)]
\item\label{infty_a} Suppose that $F\in C^k(\RR;X)$ and that
\begin{equation}\label{Ck_dom_fun}
\|F^{(j)}(s)\|\le C M(|s|)^{j+1},\quad s\in\RR,\;0\le j\le k ,
\end{equation}
for some constant $C>0$. Then, for any $c>0,$ 
\begin{equation}\label{Ck_est}
\|f(t)\|=O\left(\frac{1}{M_k^{-1}(ct)}\right),\quad t\to\infty,
\end{equation}
where $M_k^{-1}$ is the inverse of the function $M_k$ defined in \eqref{M_k}.
\item\label{infty_b} Suppose that $F\in C^\infty (\RR;X)$ and that
\begin{equation}\label{dom_fun}
\|F^{(j)}(s)\|\le C j!M(|s|)^{j+1},\quad s\in\RR,\;j\ge0,
\end{equation}
for some constant $C>0$. Then, for any $c\in (0,1/2)$,
\begin{equation}\label{BD}
\|f(t)\|=O\left(\frac{1}{\Mlog^{-1}(ct)}\right),\quad t\to\infty,
\end{equation}
where $\Mlog^{-1}$ is the inverse of the function $\Mlog$ defined in \eqref{M_log}.
\end{enumerate}
\end{thm}

\begin{rem}\label{hol_rem}
If $F\in C^\infty (\RR ; X)$ and the estimate \eqref{dom_fun} holds, then $F$ has a holomorphic extension to the region $\{\lambda\in\CC:|\I\lambda|<M(|\R\lambda|)^{-1}\}$.  Furthermore, if $F$ is the boundary function of $\smash{\widehat{f}}$ for some $f\in L^\infty(\RR_+;X)$, then $\smash{\widehat{f}}$ extends to the region $\{\lambda\in\CC:\R\lambda>-M(|\I\lambda|)^{-1}\}$, by an application of the `edge-of-the-wedge theorem'; see for instance \cite[\S2 Theorem~B]{Ru71}.
\end{rem}

The proof of Theorem~\ref{Ingham_infty} requires the following elementary lemma.

\begin{lem}\label{Leibniz}
Suppose that $\varphi:(0,\infty)\to\RR_+$ is a decreasing continuous function such that 
$\int_t^\infty\varphi(s)\,\dd s<\infty$ for all $t>0$. Given $\alpha>0$, let the function $\Phi_\alpha:(0,\infty)\to\RR$ be defined by 
$$\Phi_\alpha(t)=\int_t^\infty \varphi(s)\cos(\alpha s)\,\dd s, \quad t>0.$$
Then $|\Phi_\alpha(t)|\leq \frac{4}{\alpha}\varphi(t)$ for all $t>0$, and in particular $\int_t^\infty|\Phi_\alpha(s)|\,\dd s<\infty$ for all $t>0$.
\end{lem}

\begin{proof}
Let $\alpha>0$. Given $t>0$, let $n_0$ be the least odd integer $n\in\NN$ such that $\frac{n\pi}{2\alpha}\ge t$ and set $t_0=\frac{n_0\pi}{2\alpha}$. Furthermore, for $j\ge0$, let $t_j=\frac{j\pi}{\alpha}$.  Since $\varphi$ is decreasing and  $\varphi(s)\to0$ as $s\to\infty$ by the integrability assumption on $\varphi$,
$$\begin{aligned}\Phi_\alpha(t)&\le \int_t^{t_0}\varphi(s)|\cos(\alpha s)|\,\dd s\\&
\qquad+\sum_{j=0}^\infty \bigg(\int_{t_0+t_{2j}}^{t_0+t_{2j+1}}-\int_{t_0+t_{2j+1}}^{t_0+t_{2j+2}}\bigg)\varphi(s)|\cos(\alpha s)|\,\dd s\\&
\le \varphi(t)\int_t^{t_0}|\cos(\alpha s)|\,\dd s+\frac{2}{\alpha}\sum_{j=0}^\infty\big(\varphi\left(t_0+t_{2j} \right)-\varphi\left(t_0+t_{2j+2} \right)\big)\\&
\le \frac{2}{\alpha}\big(\varphi(t)+\varphi(t_0)\big),
\end{aligned}$$
and hence $\Phi_\alpha(t)\leq \frac{4}{\alpha}\varphi(t)$. An analogous argument shows that $\Phi_\alpha(t)\ge -\frac{4}{\alpha}\varphi(t)$, $t>0$. Thus $|\Phi_\alpha(t)|\leq \frac{4}{\alpha}\varphi(t)$ for all $t>0$, as required.
\end{proof}

\begin{proof}[Proof of Theorem \ref{Ingham_infty}]
The first statement follows from Theorem~\ref{Ingham}, so it remains only to establish the quantified statements in parts \eqref{infty_a} and \eqref{infty_b}.
Let $\phi$ be an element of the Schwartz class $\mathcal{S}(\RR)$ whose Fourier transform $\psi:=\F \phi$ satisfies $\psi(s)=1$ for $|s|\le1/2$, $0\le\psi(s)\le1$ for $1/2\le|s|\le1$ and $\psi(s)=0$ for $|s|\ge1$. For $r>0$, let $\phi_r(t)=r\phi(rt)$ $(t\in\RR)$ so that $\psi_r:=\F\phi_r$ satisfies $\psi_r(s)=\psi(s/r)$, $s\in\RR$.

Let $R>0$. Since $\int_\RR\phi(t)\,\dd t=\psi(0)=1$, 
$$f(t)-f*\phi_{R}(t)=\left(\int_{-\infty}^0+\int_0^{Rt}+\int_{Rt}^\infty\right)\big(f(t)-f(t-s/R)\big)\phi(s)\,\dd s,\quad t\ge0,$$
where $f$ has been extended by zero to $\RR$. For $t\in\RR_{\pm}$, respectively, let 
\begin{equation*}\label{Phi}
\Phi_+(t)=-\int_t^\infty\phi(s)\,\dd s\quad\mbox{and}\quad\Phi_-(t)=\int_{-\infty}^t\phi(s)\,\dd s,
\end{equation*}
 noting that $\Phi_\pm$ are both primitives of $\phi$ and that $\Phi_\pm\in L^1(\RR_\pm)$ since $\phi\in \mathcal{S}(\RR)$. Now
$$\begin{aligned}
\left\|\int_{-\infty}^0\big(f(t)-f(t-s/R)\big)\phi(s)\,\dd s\right\|&\leq \bigg\|\Big[\big(f(t)-f(t-s/R)\big)\Phi_-(s)\Big]_{s=-\infty}^{s=0}\bigg\|\\&\qquad+\left\|\int_{t}^\infty\Phi_-(R(t-s))\,\dd f(s)\right\|\\&\leq\frac{\|\Phi_-\|_{L^1(\RR_-)}\|f\|_\Lip}{R}
\end{aligned}$$
by properties of the Riemann-Stieltjes integral; see for instance \cite[Section~1.9]{ABHN11}. Similarly, using the fact that $|\phi(t)|\lesssim t^{-2}$ and hence $|\Phi_+(t)|\lesssim t^{-1}$ for all $t>0$,
$$\begin{aligned}
\bigg\|\int_0^{Rt}\big(f(t)-f(t-s/R)\big)\phi(s)\,\dd s\bigg\|&\leq \big\|\big(f(t)-f(0)\big)\Phi_+(Rt)\|\\&\qquad+\left\|\int_{0}^t\Phi_+(R(t-s))\,\dd f(s)\right\|\\&\lesssim\frac{\|f\|_\infty}{Rt}+\frac{\|\Phi_+\|_{L^1(\RR_+)}\|f\|_\Lip}{R}.
\end{aligned}$$
Here and in what follows, the statement $p\lesssim q$ for real numbers $p$ and $q$ indicates that $p\leq C q$ for some constant $C>0$ which is independent of all the parameters that are free to vary in the given situation, in this case of $R$ and $t$. Finally,
$$\bigg\|\int_{Rt}^\infty f(t)\phi(s)\,\dd s\bigg\|\le \|f\|_\infty\int_{Rt}^\infty |\phi(s)|\,\dd s\lesssim\frac{\|f\|_\infty}{Rt}.$$
Thus 
\begin{equation}\label{f-conv}
\|f(t)-f*\phi_{R}(t)\|\lesssim \frac{1}{R},\quad t\ge1.
\end{equation}

Since $\supp \psi_{R}=[-R,R]$ for each $R>0$, a calculation as in \eqref{Parseval} using the fact that $F$ is a boundary function for $\smash{\widehat{f}}$ shows that 
\begin{equation}\label{integral}
f*\phi_{R}(t)=\frac{1}{2\pi}\int_{-R}^R \e^{\ii s t}F(s)\psi_{R}(s)\,\dd s, \quad t\geq0.
\end{equation}
Let $k\ge1$ and suppose that $F\in C^k(\RR;X)$. Integrating by parts $k$ times and estimating crudely by means of \eqref{Ck_dom_fun}, it follows that
$$\begin{aligned}
\|f*\phi_{R}(t)\|&\lesssim \frac{1}{ t^k}\sum_{j=0}^k\frac{1}{R^{k-j}}\int_{-R}^R \|F^{(j)}(s)\|\,\dd s\\&\lesssim \frac{1}{ t^k}\sum_{j=0}^k\frac{RM(R)^{j+1}}{R^{k-j}}\\&\lesssim \frac{RM(R)^{k+1}}{t^k}
\end{aligned}$$
for all $t>0$ and all $R\ge1$. Combining this estimate with \eqref{f-conv} gives
$$\|f(t)\|\lesssim \frac{1}{R}+ \frac{RM(R)^{k+1}}{t^k},\quad R,t\ge1,$$
and \eqref{Ck_est} follows on setting $R=M_k^{-1}(ct)$ for $t\ge1$ sufficiently large.

Now suppose that $F\in C^\infty(\RR;X)$. In order to obtain \eqref{BD} it is necessary to make an explicit choice of the function $\phi$. Thus let $\phi\in L^1(\RR)$ be given by $\phi(0)=3/2$ and
\begin{equation}\label{phi}
\phi(t)=4\frac{\cos(t/2)-\cos(t)}{t^2},\quad t\ne0.
\end{equation}
 Then the Fourier transform $\psi$ of $\phi$ satisfies $\psi(s)=1$ for $|s|\leq 1/2$, $\psi(s)=0$ for $|s|\geq 1$ and $\psi(s)=1-|s|$ for $1/2\leq|s|\leq 1$. For $r>0$, let $\phi_r$ and $\psi_r$ be defined as before. The estimate \eqref{f-conv} still holds and follows by the same argument, except that the fact that $\Phi_\pm\in L^1(\RR_\pm)$ is now a consequence of Lemma~\ref{Leibniz}. Moreover, \eqref{integral} still holds, and integrating by parts $k$ times gives
 $$\begin{aligned}
f*\phi_{R}(t)&=\frac{(-1)^k}{2\pi(\ii t)^k}\int_{-R}^R\e^{\ii st}F^{(k)}(s)\psi_R(s)\,\dd s\\&\qquad+k\frac{(-1)^{k+1}}{2R\pi(\ii t)^k}\int_{\frac{R}{2}\leq|s|\leq R}\e^{\ii st} F^{(k-1)}(s)\sgn(s)\,\dd s\\&\qquad+(k-1)\sum_{j=0}^{k-2}\frac{(-1)^j}{2R\pi (\ii t)^{j+2}}\left[\e^{\ii st}F^{(j)}(s)-\e^{-\ii st}F^{(j)}(-s)\right]_{R/2}^R
\end{aligned}$$
for all $R,t>0$ and  $k\ge1$. For $R\geq1$, it follows that 
\begin{equation}\label{triangle}
\|f*\phi_{R}(t)\|\lesssim k! \frac{RM(R)^{k+1}}{t^k}+k\frac{M(R)}{Rt^2}\sum_{j=0}^{k-2}j!\frac{M(R)^j}{t^j},
\end{equation}
 the contribution from the second integral being dominated by that from the first. Let $k_0=\lfloor{2ct}/{M(R)}\rfloor$ and set $k=k_0$ in \eqref{triangle}. Since $k!\lesssim (k/2c\e)^k$ for all $k\geq0$ by Stirling's formula, the first term on the right-hand side can be estimated as  
 $$k_0! \frac{RM(R)^{k_0+1}}{t^{k_0}}\lesssim RM(R)\left(\frac{k_0M(R)}{2c\e t}\right)^{k_0}\lesssim RM(R)\e^{-2ct/M(R)}.$$ 
 By estimating the ratio of consecutive terms in the sum, it follows from the definition of $k_0$ that the second term on the right-hand side in \eqref{triangle} satisfies
 $$k_0\frac{M(R)}{Rt^2}\sum_{j=0}^{k_0-2}j!\frac{M(R)^j}{t^j}\le k_0\frac{M(R)}{Rt^2} \sum_{j=0}^{k_0-2} (2c)^j \lesssim k_0\frac{M(R)}{Rt^2}\le \frac{1}{Rt}.$$
 Thus combining \eqref{f-conv} and \eqref{triangle} with $k=k_0$ gives
$$\|f(t)\|\lesssim  \frac{1}{R}\left((1+R)^2M(R)^2\e^{-2ct/M(R)}+1\right)$$
for all $R,t\geq1$. Setting $R=\Mlog^{-1}(ct)$ for sufficiently large $t\ge1$ gives \eqref{BD}, thus completing the proof.
\end{proof}

\begin{rem}\label{infty_rem}
\begin{enumerate}[(a)]
\item The estimates in the above proof suggest that, instead of \eqref{Ck_dom_fun} and \eqref{dom_fun}, it might be more natural to require control over certain integral norms of $F$. For instance, replacing \eqref{dom_fun} with  
$$\int_{-R}^R\|F^{(j)}(s)\|\,\dd s\le C j! M(R)^{j+1},\quad R\ge0,\;j\ge0,$$
leads to \eqref{BD} for any $c\in(0,1)$. 
\item The choice of $k_0$ just after the estimate \eqref{triangle} is motivated by the fact that, for any constant $C>0$, the function $t\mapsto (Ct)^t$ ($t>0$) attains is global minimum at $t=(C\e)^{-1}$.
\item A simple estimate using Cauchy's integral formula shows that condition \eqref{dom_fun} is satisfied if the Laplace transform $\smash{\widehat{f}}$ of $f$ extends holomorphically to the region $\Omega_M:=\{\lambda\in\CC:\R\lambda\ge-M(|\I\lambda|)^{-1}\}$ and satisfies $|\smash{\widehat{f}}(\lambda)|\le M(|\I\lambda|)$ for all $\lambda\in\Omega_M$. This establishes the connection with the results in \cite[Section~4]{BD08}.
\item \label{fudge}There is some freedom in the choice of the function $\phi$ in the above proof. Essentially the same argument works, in particular, for the choice 
$$\phi(t)=\frac{4}{t^2}\left(\frac{\sin(t)}{t}-\cos(t)\right),\quad t\in\RR,$$ 
and in this case $\psi_R(s)=1-\frac{s^2}{R^2}$ for $|s|\leq R$ and $\psi_R(s)=0$ otherwise, so $\psi_R$ is precisely the ``fudge factor'' appearing in \cite{BD08}.
\end{enumerate}
\end{rem}

The following example illustrates how the quality of the estimates in Theorem~\ref{Ingham_infty} tends to improve with the smoothness of the boundary function $F$ if the function $M$ grows only moderately fast, but that the rate of decay can become independent of the smoothness of $F$ if $M$ grows very rapidly.

\begin{ex}\label{ex_infty}
Let $f:\RR_+\to X$ and $F:\RR\to X$ be as in the statement of Theorem~\ref{Ingham_infty} and consider the function  $M:\RR_+\to[1,\infty)$  given by $M(R)=(1+R)^\alpha$ for some $\alpha>0$. If $F\in C^k(\RR;X)$ for some $k\ge1$ and \eqref{Ck_dom_fun} holds, then \eqref{Ck_est} becomes 
$$\|f(t)\|=O\left(t^{-\frac{k}{\alpha(k+1)+2}}\right),\quad t\to\infty,$$
and if $F\in C^\infty(\RR;X)$ and \eqref{dom_fun}, \eqref{BD} becomes 
$$\|f(t)\|=O\left(\left(\frac{\log t}{t}\right)^{1/\alpha}\right),\quad t\to\infty.$$
Thus the rate of decay improves with the smoothness of $F$.

On the other hand, if $M(R)=\exp(R^\alpha)$ for some $\alpha>0$, then \eqref{Ck_est} for any $k\ge1$ and \eqref{BD} all become
$$\|f(t)\|=O\left((\log t)^{-1/\alpha}\right),\quad t\to\infty,$$
so the quality of the estimate is independent of the smoothness of $F$.
\end{ex}

Corollary~\ref{cor:BD} below provides a uniform rate of decay for smooth orbits of bounded $C_0$-semigroups whose generator has no spectral points on the imaginary axis. The result was first obtained in \cite{BD08} by means of a contour integral argument involving the ``fudge factor'' appearing in Korevaar's proof of Ingham's theorem; see Remark~\ref*{infty_rem}.\eqref{fudge} above. Here the result is a straightforward consequence of Theorem~\ref{Ingham_infty}.

\begin{cor}\label{cor:BD}
Let $T$ be a bounded $C_0$-semigroup with generator $A$ on a Banach space $X$ and suppose that $\sigma(A)\cap\ii\RR=\emptyset$. If $M:\RR_+\to[1,\infty)$ is a continuous non-decreasing function such that 
$$\|R(\ii s,A)\|\le M(|s|),\quad s\in\RR,$$
 then, for any $c\in(0,1/2)$,  
$$\|T(t)A^{-1}\|=O\left(\frac{1}{\Mlog^{-1}(ct)}\right),\quad t\to\infty.$$
\end{cor}

\begin{proof}
For each $x\in X$, the function $f_x\in\BUC(\RR_+;X)$ given by $f_x(t)=T(t)A^{-1}x$ $(t\geq0)$ satisfies the assumptions of Theorem~\ref{Ingham_infty} for $F_x(s)=R(\ii s,A)A^{-1}x$, $s\in\RR$.  Moreover, 
$$F^{(k)}_x(s)=(-\ii)^kk!R(\ii s,A)^{k+1}A^{-1}x$$  
 for all $k\geq0$ and $s\in \RR$,  so \eqref{dom_fun} holds. Thus, for each $c\in(0,1/2)$ and any $x\in X$ with $\|x\|=1$,  $\|T(t)A^{-1}x\|=O(\Mlog^{-1}(ct)^{-1})$ as $t\to\infty$, and the implicit constant is independent of $x$. Hence the result follows.
\end{proof}

\begin{rem}
It is shown in \cite{BT10} that Corollary~\ref{cor:BD} is optimal in the case where $M(R)=C(1+R)^\alpha$  for some constants $C\ge1$, $\alpha>0$. It follows, in particular, that part \eqref{infty_b} of Theorem~\ref{Ingham_infty} is optimal.
\end{rem}

\subsection{Singularity at zero} \label{zero}

In this section we  obtain a quantified version of Theorem~\ref{Ingham} in the case where $\Sigma=\{0\}$. In this case it is necessary to impose an additional ergodicity assumption, namely that $f$ has bounded primitive; see also Remark~\ref{rem_zero} below. The rate of decay of $f(t)$ as $t\to\infty$  will be determined by the rate of growth of the boundary function $F(s)$ as $|s|\to0$.  As the proof of Theorem~\ref{Ingham_infty} shows, it is in general necessary also to take into consideration the growth of $F(s)$ as $|s|\to\infty$. This will be done in Section~\ref{zero_infty} below, but here we make an additional assumption on the function $f$ which ensures that the  behaviour of $F(s)$ for large values of $|s|$ can be neglected. Thus a function $f\in \Cb(\RR_+;X)$ will be said to be \emph{asymptotically regular} if
\begin{equation}\label{extra}
\|f(t)-f*\phi(t)\|\le\frac{C}{t},\quad t>0,
\end{equation}
for some $\phi\in L^1(\RR)$ such that  $\F\phi\in\Wc(\RR):=W^{1,1}(\RR)\cap\Cc(\RR)$ and $\F\phi\equiv1$ near 0. The following simple lemma shows that if $f$ is asymptotically regular then \eqref{extra} is in fact satisfied for \emph{all} suitable $\phi\in L^1(\RR)$ provided that $F\in \smash{W^{1,1}_\mathrm{loc}}(\RR\backslash\{0\};X)$.

\begin{lem}\label{regular}
Let $X$ be a Banach space and suppose that $f\in\Cb(\RR_+;X)$ is asymptotically regular. If $\smash{\widehat{f}}$ admits a boundary function $F\in \smash{W^{1,1}_\mathrm{loc}}(\RR\backslash\{0\};X)$, then \eqref{extra} holds for all $\phi\in L^1(\RR)$ such that $\F\phi\in \smash{\Wc}(\RR)$ and $\F\phi\equiv1$  in a neighbourhood of 0.
\end{lem}

\begin{proof}
If $\phi_1$, $\phi_2\in L^1(\RR)$ are such that $\F\phi_1$, $\F\phi_2\in \Wc (\RR)$ and $\F\phi_1$, $\F\phi_2\equiv 1$ in a neighbourhood of $0$, then $\F \phi_1 -\F\phi_2\in \Wc (\RR)$ has compact support in $\RR\backslash\{ 0\}$ and since $F$ has a locally integrable derivative the product $F\, (\F\phi_1 - \F\phi_2)$ lies in $\Wc (\RR ;X)$. Integrating by parts once and using the fact that   $F$ coincides with the distributional Fourier transform of $f$ on $\RR\backslash\{0\}$, it follows that  
$$\|f*(\phi_1 - \phi_2)(t)\|=\big\|\F^{-1}(F \, (\F\phi_1 - \F\phi_2)) (t)\big\| \lesssim \frac{1}{t},\quad t>0,$$
as required.
\end{proof}

Theorem~\ref{Ingham_zero} below, the main result of this section, is an analogue of Theorem~\ref{Ingham_infty} in the setting where the boundary function $F$ is allowed to have a singularity at zero.  Given a continuous non-increasing function $m:(0,1]\to[1,\infty)$, define the map $m_k:(0,1]\to(0,\infty)$ for each $k\ge1$ by
\begin{equation}\label{mk}
m_k(r)=m(r)\left(\frac{m(r)}{r}\right)^{1/k},\quad 0<r\le1,
\end{equation}
and define $\mlog:(0,1]\to(0,\infty)$ by 
\begin{equation}\label{mlog}
\mlog(r)=m(r)\log\left(1+\frac{m(r)}{r}\right),\quad 0<r\leq1.
\end{equation}

\begin{thm}\label{Ingham_zero}
Let $X$ be a Banach space and let $f\in \Cb(\RR_+;X)$ be an asymptotically regular function such that 
\begin{equation}\label{int_bd}
\sup_{t\geq0}\left\|\int_0^tf(s)\,\dd s\right\|<\infty.
\end{equation}
Suppose that $\smash{\widehat{f}}$ admits a boundary function $F\in\Lloc(\RR\backslash\{0\};X)$. Then $f\in C_0(\RR_+;X)$. Moreover, given a continuous non-increasing function $m:(0,1]\to[1,\infty)$, the following hold.
\begin{enumerate}[(a)]
\item Suppose that $F\in C^k(\RR\backslash\{0\};X)$ and that
\begin{equation}\label{Ck_dom_fun_zero}
\|F^{(j)}(s)\|\le C |s|^{k-j}m(|s|)^{k+1},\quad 0<|s|\le1,\;0\le j\le k,
\end{equation}
for some constant $C>0$. Then, for any $c>0$,
\begin{equation}\label{Ck_est_zero}
\|f(t)\|=O\left(m_k^{-1}(ct)+\frac{1}{t}\right),\quad t\to\infty,
\end{equation}
where $m_k^{-1}$ is the inverse of the function $m_k$ defined in \eqref{mk}.
\item Suppose that $F\in C^\infty (\RR\backslash\{0\};X)$ and that
\begin{equation}\label{dom_fun_zero}
\|F^{(j)}(s)\|\le C j!|s|m(|s|)^{j+1},\quad 0<|s|\le1,\;j\ge0,
\end{equation}
for some constant $C>0$. Then, for any $c\in (0,1)$,
\begin{equation}\label{BCT}
\|f(t)\|=O\left(\mlog^{-1}(ct)+\frac{1}{t}\right),\quad t\to\infty,
\end{equation}
where $\mlog^{-1}$ is the inverse of the function $\mlog$ defined in \eqref{mlog}.
\end{enumerate}
\end{thm}

\begin{proof}
Let $\phi\in L^1(\RR)$ be such that $\F\phi\in \Wc (\RR)$ and $\F\phi\equiv 1$ in a neighbourhood of $0$, and for $r>0$ let $\phi_r(t)=r\phi(rt)$, $t\in\RR$. Moreover, let $g(t)=\smash{\int_0^t} f(s)\,\dd s$ ($t\geq0$) denote the primitive of $f$. By assumption \eqref{int_bd}, integrating by parts gives 
\begin{equation}\label{r}
\begin{aligned}
\|f*\phi_r(t)\|&=r\left\|\int_0^\infty g(s/r)\phi'(rt-s)\,\dd s\right\|\lesssim r,\quad t\ge0.
\end{aligned}
\end{equation}
Choose $\phi\in L^1(\RR)$ as above in such a way that the Fourier transform $\psi:=\F\phi$ is a smooth function satisfying $\psi(s)=0$ for $|s|\ge1$, $\psi(s)=1$ for $|s|\le1/2$, and $0\le\psi(s)\le$ for $1/2\le|s|\le1$. For $r\in(0,1/4]$, let $\psi_r:=\F\phi_r$, where $\phi_r\in L^1(\RR)$ is as defined above. Then $\psi_r(s)=\psi(s/r)$ for $s\in \RR$ and, since $F$ is a boundary function for $\smash{\widehat{f}}$, a calculation as in \eqref{Parseval} shows that
\begin{equation}\label{f*diff}
f*(\phi-\phi_{2r})(t)=\frac{1}{2\pi}\int_{r\leq|s|\leq 1} \e^{\ii st} F(s) \big(\psi(s)-\psi_r(s)\big)\,\dd s,\quad t\geq0,
\end{equation}
and hence $f*(\phi-\phi_{2r})\in C_0(\RR_+;X)$ by the Riemann-Lebesgue lemma. Since \eqref{r} holds for arbitrarily small $r\in(0,1/4]$, it follows from the fact that $f$ is asymptotically regular that in fact $f\in C_0(\RR_+;X)$.

Suppose that $F\in C^k(\RR\backslash\{0\};X)$ for some $k\ge1$. Integrating by parts $k$ times in \eqref{f*diff} and estimating crudely by means of \eqref{Ck_dom_fun_zero} gives
\begin{equation}\label{m(r)}
\|f*(\phi-\phi_r)(t)\|\lesssim \frac{m(r)^{k+1}}{t^k},\quad t>0.
\end{equation}
By asymptotic regularity of $f$, the estimate \eqref{Ck_est_zero} follows from \eqref{r} and \eqref{m(r)} on setting $r=m_k^{-1}(ct)$ for sufficiently large $t>0$.

Now suppose that $F\in C^\infty(\RR\backslash\{0\};X)$. In order to obtain \eqref{BCT}, let $\phi\in L^1(\RR)$ be defined as in \eqref{phi}. With the same notation as above, \eqref{f*diff} again holds. Integrating by parts $k$ times and estimating gives 
$$\begin{aligned}
\|f*(\phi-\phi_{2r})(t)\|&\lesssim \frac{1}{t^k}\int_{2r\leq|s|\leq\frac{1}{2}}\|F^{(k)}(s)\|\,\dd s\\&\qquad+\frac{k}{t^k}\left(\frac{1}{r}\int_{r\leq|s|\leq2r}+\int_{\frac{1}{2}\leq|s|\leq1}\right)\|F^{(k-1)}(s)\|\,\dd s\\&\qquad+k\sum_{j=0}^{k-2}\frac{j!}{t^{j+2}}\left(m(r)^{j+1}+m(1/2)^{j+1}\right)
\end{aligned}$$
for all $t>0$. Setting $k=\lfloor{ct}/{m(r)}\rfloor$, it follows as in the proof of Theorem~\ref{Ingham_infty} that
\begin{equation}\label{conv_est}
\|f*(\phi-\phi_{2r})(t)\|\lesssim m(r)\e^{-ct/m(r)}+\frac{1}{t}
\end{equation}
 for all $t>0$ and $r\in(0,1/4]$. Not set $r=\mlog^{-1}(ct)$ for sufficiently large $t>0$.  Then
 $$m(r)\e^{-ct/m(r)}=\frac{rm(r)}{r+m(r)}\le r,$$
and the result follows from \eqref{r} and asymptotic regularity of $f$.
\end{proof}

\begin{rem}\label{rem_zero}
\begin{enumerate}[(a)]
\item Theorems \ref{Ingham_infty} and \ref{Ingham_zero} can be regarded as dual to one another in a certain sense.  In particular,  the condition $f\in \Lip(\RR_+;X)$ in the former, which is close to requiring a bounded derivative, is replaced in the latter by condition \eqref{int_bd}, requiring $f$ to have a bounded primitive. Note that the latter assumption implies that $f$ is \emph{uniformly mean ergodic} with mean $0$ in the sense that
$$\lim_{\lambda\to 0+} \lambda \int_0^\infty e^{-\lambda s} f(t+s) \; \dd s = 0$$
uniformly for $t\ge0$. This condition is required in Ingham's theorem in the case when $F$ has a singularity at $0$; see \cite[Chapter~4]{ABHN11}.
\item As in \cite{Ma11}, it is possible to extend Theorem~\ref{Ingham_zero} to the case where $F\in \Lloc(\RR\backslash \Sigma;X)$ for an arbitrary finite set $\Sigma\subset\RR$ of singularities. The case  $\Sigma=\emptyset$ is included here but is of no real interest.
\end{enumerate}
\end{rem}

The following example is analogous to Example~\ref{ex_infty}.

\begin{ex}
Let $f:\RR_+\to X$ and $F:\RR\backslash\{0\}\to X$ be as in the statement of Theorem~\ref{Ingham_zero}, and consider the function  $m:(0,1]\to[1,\infty)$  given by $m(r)=r^{-\alpha}$ for some $\alpha\ge1$. If $F\in C^k(\RR\backslash\{0\};X)$ for some $k\ge1$ and \eqref{Ck_dom_fun_zero} holds, then \eqref{Ck_est_zero} becomes 
$$\|f(t)\|=O\left(t^{-\frac{k}{\alpha(k+1)+1}}\right),\quad t\to\infty,$$
and if $F\in C^\infty(\RR\backslash\{0\};X)$ and \eqref{dom_fun_zero} holds, \eqref{BCT} becomes 
$$\|f(t)\|=O\left(\left(\frac{\log t}{t}\right)^{1/\alpha}\right),\quad t\to\infty.$$
Thus the rate of decay improves with the smoothness of $F$.

On the other hand, if $m(r)=\exp(r^{-\alpha})$ for some $\alpha>0$, then \eqref{Ck_est_zero} for any $k\ge1$ and \eqref{BCT} all become
$$\|f(t)\|=O\left((\log t)^{-1/\alpha}\right),\quad t\to\infty,$$
so the quality of the estimate in this case, where $m(r)\to\infty$ very rapidly as $r\to0+$,  is independent of the smoothness of $F$.
\end{ex}

The following result is a direct consequence of Theorem~\ref{Ingham_zero} and is an improvement of \cite[Theorem~6.15]{BCT}. Recall from \cite{BaSr03} the definition of the {\em non-analytic growth bound} $\zeta (T)$ for a $C_0$-semigroup $T$ on a Banach space $X$,
$$\begin{aligned}
\zeta (T)  := \inf \big\{ \omega\in\RR : &\;\| T(t)-S(t)\|\leq M \e^{\omega t}\;\mbox{for some $M\geq1$, all $t>0$ and}\\&\mbox{\;some analytic $S : \Sigma_\theta \to\B(X)$ with $0<\theta\le\pi$} \big\}.
\end{aligned}$$
Here $\Sigma_\theta := \{ z\in\CC : |\arg z|<\theta \}$, for $0<\theta\le\pi$, denotes a sector in the complex plane. It follows from properties of the Laplace transform of analytic functions that if $\zeta (T) <0$, then $\sigma(A)\cap\ii\RR$ is compact and
\begin{equation}\label{res_bd}
\sup_{|s|\ge R}\|R(\ii s,A)\|<\infty
\end{equation} 
for any sufficiently large $R>0$. For bounded $C_0$-semigroups on Hilbert spaces, this condition is even equivalent to $\zeta (T)<0$; see \cite{BaSr03}. Following the terminology of \cite{Se15}, we call a bounded $C_0$-semigroup $T$ satisfying $\zeta (T)<0$ {\em asymptotically analytic}.

\begin{cor}\label{cor:BCT}
Let $T$ be a asymptotically analytic $C_0$-semigroup with generator $A$ on a Banach space $X$ and suppose that $\sigma(A)\cap\ii\RR=\{0\}$. If $m:(0,1]\to[1,\infty)$ is a continuous non-increasing function such that
$$\|R(\ii s,A)\|\leq m(|s|),\quad 0<|s|\leq 1,$$
then, for any $\omega>0$ and $c\in(0,1)$,  
$$\|T(t)AR(\omega,A)\|=O\big(\mlog^{-1}(ct)\big),\quad t\to\infty.$$ 
\end{cor}

\begin{proof}
Let $x\in X$, $\omega>0$, and let $f_x(t)=T(t)AR(\omega,A)x$, $t\geq0$. Then the primitive $g_x$ of $f_x$ is given, for  $t\geq0$, by
$$g_x(t)=\int_0^t f_x(s)\,\dd s=A\int_0^tT(s)R(\omega,A)x\,\dd s=(T(t)-I)R(\omega,A)x,\quad t\ge0,$$
 so $g_x\in\BUC(\RR_+;X)$. Furthermore, 
$$\widehat{f_x}(\lambda)=\lambda\widehat{g_x}(\lambda)=\big(\lambda R(\lambda,A)-I\big)R(\omega,A)x,\quad \R\lambda>0,$$
 and hence \eqref{Parseval} is satisfied for the function $F=F_x$ given by 
$$F_x(s)=\big(\ii sR(\ii s,A)-I\big)R(\omega,A)x,\quad s\in\RR\backslash\{0\}.$$
Now
$$F_x^{(k)}(s)=(-\ii)^kk!\big(\ii s R(\ii s,A)-I\big)R(\ii s,A)^kR(\omega,A)x$$
for all $s\ne0$ and all integers $k\geq0$. Since $\|R(\lambda,A)\|\geq \dist(\lambda,\sigma(A))^{-1}$ for all $\lambda\in\rho(A)$, \eqref{dom_fun_zero} is satisfied.

 Let $\phi\in L^1 (\RR )$ be such that $\F\phi\in C^\infty_c (\RR )$ and $\F\phi\equiv 1$ in a neighbourhood of $0$. By an argument as in the proof of \cite[Theorem~3.4]{Se15} using \cite[Theorem 3.6]{BaSr03}, $h_x:=f_x-f_x*\phi $ is  measurable and
\begin{equation} \label{eq.1}
 \int_\RR \| h_x (t)\| \,\dd t \lesssim \|x\|,\quad x\in X,
\end{equation}
were $f_x$ has been extended by zero to $\RR$. Using  the semigroup property, 
$$\begin{aligned}
\int_0^t T(t-s) h_x(s)\,\dd s&= \int_0^t \left(f_x(t)\,\dd s - \int_{-\infty}^s  f_x(t-r)\phi (r)\,\dd r\right)\dd s \\
& = t h_x(t) + \int_0^t  f_x(t-r)r\phi (r)\,\dd r . 
\end{aligned}$$
for all $x\in X$ and $t\geq 0$ and, by \eqref{eq.1} and boundedness of $T$,
$$\left\| \int_0^t T(t-s) h_x(s)\,\dd s \right\| \lesssim \| x\|,\quad x\in X,\; t\ge0. $$
Moreover, since $\F\phi\in C^\infty_c (\RR )$, the function $r\mapsto r\phi (r)$ is integrable, and therefore
$$\left\| \int_0^t f_x(t-r) r\phi (r) \,\dd r \right\| \lesssim \| x\|,\quad x\in X,\;t\ge0. $$
Thus \eqref{extra} is satisfied. Since the implicit constants can be taken to be independent of $x\in X$ with $\|x\|=1$, \eqref{BCT} holds for the function $f(t)=T(t)AR(\omega,A)$, $t\ge0$.

To finish the proof, note that $\mlog(r)\gtrsim m(r)\geq 1/r$ for all sufficiently small $r\in(0,1]$, and hence, given $c>0$, $t^{-1}\lesssim\mlog^{-1}(ct)$ for all sufficiently large $t>0$.
\end{proof}

\subsection{Singularities at zero and infinity}\label{zero_infty}

In this brief final section we formulate a result analogous to Theorems~\ref{Ingham_infty} and \ref{Ingham_zero} in which $\Sigma=\{0\}$ but $f$ is not assumed to be asymptotically regular. Consequently, the boundary function $F(s)$ is assumed to have controlled growth both as $|s|\to0$ and as $|s|\to\infty$. The proofs are  similar to those given in Sections~\ref{infty} and \ref{zero}, and hence are omitted.

\begin{thm}\label{Ingham_zero_infty}
Let $X$ be a Banach space and let $f\in \Cb(\RR_+;X)\cap \Lip(\RR_+;X)$ be such that 
$$\sup_{t\geq0}\bigg\|\int_0^tf(s)\,\dd s\bigg\|<\infty.$$
Suppose that $\smash{\widehat{f}}$ admits a boundary function $F\in\smash{\Lloc}(\RR\backslash\{0\};X)$. Then $f\in C_0(\RR_+;X)$. Moreover, given a continuous non-increasing function $m:(0,1]\to[1,\infty)$ and a continuous non-decreasing function $M:[1,\infty)\to[1,\infty)$, the following hold.
\begin{enumerate}[(a)]
\item Suppose that $F\in C^k(\RR\backslash\{0\};X)$ for some $k\ge1$ and that
\[
 \| F^{(j)}(s)\|\le \begin{cases} 
                     C |s|^{k-j}m(|s|)^{k+1},  &\text{ $0<|s|\le1,\;0\le j\le k,$} \\[1mm]
                     C M(|s|)^{j+1}, &\text{$|s|\ge1,\;0\le j\le k,$}
                    \end{cases}
\]
for some constant $C>0$. Then, for any $c>0$,
$$\|f(t)\|=O\left(m_k^{-1}(ct)+\frac{1}{M_k^{-1}(ct)}\right),\quad t\to\infty,$$
where $m_k^{-1}$ and $M_k^{-1}$ are the inverse of the functions $m_k$ and $M_k$ defined as in \eqref{mk} and \eqref{M_k}, respectively.
\item Suppose that $F\in C^\infty (\RR\backslash\{0\};X)$ and that
\[
 \| F^{(j)}(s)\|\le \begin{cases} 
                     C j!|s|m(|s|)^{j+1},&\text{ $0<|s|\le1,\;j\ge0,$} \\[1mm]
                     Cj! M(|s|)^{j+1},&\text{ $|s|\ge1,\; j\ge0,$}
                    \end{cases}
\]
for some constant $C>0$. Then, for any $c\in (0,1/2)$,
$$\|f(t)\|=O\left(\mlog^{-1}(ct)+\frac{1}{\Mlog^{-1}(ct)}+\frac{1}{t}\right),\quad t\to\infty,$$
where $\mlog^{-1}$ and $\Mlog^{-1}$ are the inverse of the functions $\mlog$ and $\Mlog$ defined as in \eqref{mlog} and \eqref{M_log}, respectively.
\end{enumerate}
\end{thm}

The following result is an application of Theorem~\ref{Ingham_zero_infty} to $C_0$-semigroups. The same result is obtained in \cite[Proposition~3.1]{Ma11} by means of a contour integral argument; see also \cite[Section~8]{BCT}.

\begin{cor}
Let $T$ be a bounded asymptotically analytic $C_0$-semigroup with generator $A$ on a Banach space $X$ and suppose that $\sigma(A)\cap\ii\RR=\{0\}$. Suppose further that $m:(0,1]\to[1,\infty)$ is a continuous non-increasing function such that
$$\|R(\ii s,A)\|\leq m(|s|),\quad 0<|s|\leq 1,$$
and that $M:[1,\infty)\to[1,\infty)$ is a continuous non-decreasing function such that
$$\|R(\ii s,A)\| \leq M(|s|),\quad |s|\ge1.$$
Then, for any $\omega>0$ and $c\in(0,1/2)$,  
$$\|T(t)AR(\omega,A)^2\|=O\left(\mlog^{-1}(ct)+\frac{1}{\Mlog^{-1}(ct)}\right),\quad t\to\infty.$$ 
\end{cor}


\begin{thebibliography}{10}

\bibitem{ABHN11}
W.~Arendt, C.J.K. Batty, M.~Hieber, and F.~Neubrander.
\newblock {\em Vector-valued Laplace transforms and Cauchy problems}.
\newblock Birkh\"auser, Basel, second edition, 2011.

\bibitem{BCT}
C.J.K. Batty, R.~Chill, and Y.~Tomilov.
\newblock Fine scales of decay of operator semigroups.
\newblock {\em J. Eur. Math. Soc.}, 2015, to appear.

\bibitem{BD08}
C.J.K. Batty and T.~Duyckaerts.
\newblock Non-uniform stability for bounded semi-groups in {B}anach spaces.
\newblock {\em J. Evol. Equ.}, 8:765--780, 2008.

\bibitem{BaSr03}
C.J.K. Batty and S.~Srivastava.
\newblock The non-analytic growth bound of a {$C_0$}-semigroup and
  inhomogeneous {Cauchy} problems.
\newblock {\em J. Differential Equations}, 194:300--327, 2003.

\bibitem{BT10}
A.~Borichev and Y.~Tomilov.
\newblock Optimal polynomial decay of functions and operator semigroups.
\newblock {\em Math. Ann.}, 347:455--478, 2010.

\bibitem{Ch98}
R.~Chill.
\newblock Tauberian theorems for vector-valued {F}ourier and {L}aplace
  transforms.
\newblock {\em Studia Math.}, 128:55--69, 1998.

\bibitem{CT07}
R.~Chill and Y.~Tomilov.
\newblock Stability of operator semigroups: ideas and results.
\newblock In {\em Perspectives in Operator Theory}. Banach Center Publications,
  Volume 75, Polish Academy of Sciences, Warsaw, 2007.

\bibitem{Ing33}
A.E. Ingham.
\newblock {On Wiener's method in Tauberian theorems}.
\newblock {\em Proc. Lond. Math. Soc.}, S2-38(1):458--480, 1933.

\bibitem{Ko82}
J.~Korevaar.
\newblock On {N}ewman's quick way to the prime number theorem.
\newblock {\em Math. Intelligencer}, 4:108--115, 1982.

\bibitem{Ma11}
M.M. Mart\'inez.
\newblock Decay estimates of functions through singular extensions of
  vector-valued {Laplace transforms}.
\newblock {\em J. Math. Anal. Appl.}, 375:196--206, 2011.

\bibitem{Ru71}
W.~Rudin.
\newblock {\em Lectures on the edge-of-the-wedge theorem}.
\newblock American Mathematical Society, Providence, Rhode Island, 1971.

\bibitem{Se15}
D.~Seifert.
\newblock {A Katznelson-Tzafriri theorem for measures}.
\newblock {\em Integr. Equ. Oper. Theory}, 81:255--270, 2015.

\end{thebibliography}
\bibliographystyle{plain}

\end{document}